\documentclass[11pt,english,reqno]{amsart}

\usepackage[T1]{fontenc}
\usepackage[english]{babel}
\usepackage{amssymb,url,xspace,mathrsfs}
\usepackage{lmodern}
\usepackage{enumitem}
\usepackage{tikz}
\usetikzlibrary{decorations.markings}
\usetikzlibrary{automata, positioning, arrows}
\usepackage{array}

\makeatletter
\@namedef{subjclassname@2020}{\textup{2020} Mathematics Subject Classification}
\makeatother

\newtheorem{theorem}{Theorem}[section] % 1st argument is your name for it

\newtheorem{lemma}[theorem]{Lemma}     % 2nd argument is what is printed

\newtheorem{proposition}[theorem]{Proposition}

\numberwithin{equation}{section}
\theoremstyle{definition}
\newtheorem{question}{Question}
\newtheorem{example}{Example}
\newtheorem{remark}{Remark}

\newcommand{\pres}[3]{\textnormal{#1} \langle #2 \mid #3 \rangle}

\DeclareMathOperator{\pos}{\normalfont\textsc{pos}}
\DeclareMathOperator{\red}{\normalfont\textsc{red}}
\DeclareMathOperator{\cred}{\normalfont\textsc{cred}}
\DeclareMathOperator{\any}{\normalfont\textsc{any}}
\DeclareMathOperator{\ovl}{\normalfont\textsc{ovl}}
\DeclareMathOperator{\Pre}{Pre}

\DeclareMathOperator{\FP}{FP}

\title[One-relator groups and units of special inverse monoids]{On one-relator groups and units of special one-relation inverse monoids}
\date {\today}
\author{Carl-Fredrik Nyberg-Brodda}
\thanks{The author gratefully acknowledges funding from the Dame Kathleen Ollerenshaw Trust, which is supporting his research at the University of Manchester.}
\address{Department of Mathematics, Alan Turing Building, University of Manchester, United Kingdom.}
\email{carl-fredrik.nybergbrodda@manchester.ac.uk}


\begin{document}

\begin{abstract}
This note investigates and clarifies some connections between the theory of one-relator groups and \textit{special} one-relation inverse monoids, i.e. those inverse monoids with a presentation of the form $\pres{Inv}{A}{w=1}$. We show that every one-relator group admits a special one-relation inverse monoid presentation. We subsequently consider the classes $\any,\red, \cred, \pos$ of one-relator groups which can be defined by special one-relation inverse monoid presentations in which the defining word is arbitrary; reduced; cyclically reduced; or positive, respectively. We show that the inclusions $\any \supset \cred \supset \pos$ are all strict. Conditional on a natural conjecture, we prove $\any \supset \red$. Following this, we use the Benois algorithm recently devised by Gray \& Ru\v{s}kuc to produce an infinite family of special one-relation inverse monoids which exhibit similar pathological behaviour (which we term \textit{O'Haresque}) to the O'Hare monoid with respect to computing the minimal invertible pieces of the defining word. Finally, we provide a counterexample to a conjecture by Gray \& Ru\v{s}kuc that the Benois algorithm always correctly computes the minimal invertible pieces of a special one-relation inverse monoid.
\end{abstract}

%\begin{altabstract}
%English abstract
%\end{altabstract}
\maketitle

%\tableofcontents

\section{Introduction}

\noindent The explosive development of combinatorial group theory as a central part of combinatorial algebra was one of the great mathematical successes of the 20th century. Ever since its genesis in the work by Dyck \cite{Dyck1882} and Dehn \cite{Dehn1911, Dehn1914} around the turn of the century, its numerous connections with logic, decidability theory, formal language theory, combinatorics -- among many others -- as well as its spectacularly successful offshoot in \textit{geometric} group theory, have all served to solidify its place as a foundational part of modern research in group theory. However, the importance of combinatorial \textit{semigroup} theory for this development is often somewhat unduly overlooked. Working around the same time as Dehn, Thue \cite{Thue1914} initiated in 1914 what would become combinatorial semigroup theory. The word problem for semigroups took centre-stage; Church \cite{Church1937} proved this problem to be undecidable in the finitely generated case, and Markov \cite{Markov1947a, Markov1947b} and Post \cite{Post1947} subsequently improved this to the finitely presented case. This discovery of undecidable problems in combinatorial semigroup theory would spur further research into producing similar results for groups. All subsequent examples of undecidable problems produced for groups at this time were directly dependent on the existence of undecidable problems for semigroups. For example, the first example(s) of an undecidable problem for a finitely presented group, i.e. the constructions by Novikov \cite{Novikov1952} and Boone \cite{Boone1958} of a finitely presented group with undecidable word problem, were both directly based on the construction by Turing \cite{Turing1950} of a finitely presented cancellative semigroup with undecidable word problem, which in turn was based on the aforementioned result by Markov and Post.\footnote{Turing's paper contained many issues, some of which were fixed by Boone \cite{Boone1958b}. Novikov \& Adian \cite{Novikov1958b} later published a separate proof of Turing's result to avoid these issues.} Similarly, the Adian-Rabin theorem \cite{Adian1955, Rabin1958} of the unrecognisability of ``Markov properties'' of finitely presented groups grew out -- as the name suggests! -- of the proof of the same result for finitely presented semigroups by Markov \cite{Markov1951}.

Thus the theory of combinatorial semigroup theory has played, and continues to play, an important r\^ole in combinatorial group theory. Recently, efforts have been made to further understand the foundations of combinatorial \textit{inverse} semigroup theory. An \textit{inverse} monoid is a monoid $M$ such that for every $x \in M$, there exists a unique $y \in M$ such that $xyx = x$ and $yxy = y$. In particular, every group is an inverse monoid; thus the class of inverse monoids forms an intermediate class between groups and monoids. Inverse monoids, much as monoids and groups, admit presentations, and questions one can ask in combinatorial group and semigroup theory can therefore also be asked about inverse monoids. Much exotic behaviour has been discovered, but few fully general results, and so an exhaustive theory to serve as a foundation still appears rather elusive. The object of this article is to shed some light on a small corner of this foundation. Before we can do this, however, it is important to understand what kinds of results one might expect to form part of this foundation.

As part of the origins of a program to solve the word problem for all one-relation monoids $\pres{Mon}{A}{u=v}$ -- a problem which remains open even today, see the recent survey \cite{NybergBrodda2021b} -- Adian \cite{Adian1960} identified the \textit{special} case $\pres{Mon}{A}{w=1}$ as particularly amenable to study. In particular, he solved the word problem for any such monoid. The proof first shows that this word problem reduces to the same problem for the group of units of the monoid in question, and then demonstrates that this group is a one-relator group $\pres{Gp}{B}{r=1}$. As the word problem is since 1932 well-known to be decidable for one-relator groups by Magnus \cite{Magnus1932}, this yields the result. An important fact is that the class of one-relator groups which can appear as the group of units of a special one-relation monoid is strictly smaller than the class of all one-relator groups. Specifically, Adian proved that the group of units of $\pres{Mon}{A}{w=1}$ is a one-relator group $\pres{Gp}{B}{r=1}$, where $r \in B^\ast$ is some \textit{positive} word. One-relator groups which admit some such presentation are called \textit{positive}. Though not much is known about this class in generality, Baumslag \cite{Baumslag1971} proved that all positive one-relator groups are residually solvable. It is not hard to see, but requires more than trivial arguments, that the free abelian $\mathbb{Z}^2 = \pres{Gp}{a,b}{aba^{-1}b^{-1}=1}$ is not a positive one-relator group. This result was proved by Magnus already in 1930, using the \textit{Hauptform des Freiheitssatzes} \cite[\S6, p. 159]{Magnus1930}, see Otto \cite{Otto1988} for a short proof. In particular, the group of units of a special one-relation monoid is never isomorphic to $\mathbb{Z}^2$. On the other hand, the fundamental group $\Pi = \pres{Gp}{a,b}{abab^{-1} = 1}$ of the Klein bottle \textit{is} a positive one-relator group, as $\Pi \cong \pres{Gp}{c,d}{c^2d^2 = 1}$. As $\mathbb{Z}^2$ embeds as an index $2$ subgroup into $\Pi$, it follows that positivity is not a virtual property, nor one that is inherited by subgroups; it is thus a rather elusive property of one-relator groups.\footnote{To the best of the author's knowledge, it is not known whether one can recognise if a one-relator group is positive or not; that is, the isomorphism problem for positive one-relator groups relative to the class of one-relator groups seems to be an open problem.}

The above paragraph shows that there is already some interplay between combinatorial semigroup and group theory. One might thus like to inquire as to the \textit{monoid}-theoretic properties of certain one-relator groups. For example, one might ask: which one-relator groups admit a presentation of the form $\pres{Mon}{A}{w=1}$? Perrin \& Schupp \cite{Perrin1984} answered this: a one-relator group admits a presentation $\pres{Mon}{A}{w=1}$ if and only if it is a positive one-relator group (the ``only if'' direction is trivial). As an example of this, the group $\Pi$ above is isomorphic to $\pres{Mon}{a,b}{abba=1}$. Thus we have an instance of a monoid-theoretic problem being reduced to a group-theoretic problem. It is natural to ask a similar question for special one-relator \textit{inverse} monoids; the purpose of this note is to give some words in response to this question, following recent fundamental investigations by Gray \& Ru\v{s}kuc. We first give some indication why working with special one-relation inverse monoids is significantly harder than the ordinary monoid case. 

A principal issue lies in the fact that \textit{free} inverse monoids (i.e. special zero-relation inverse monoids) are already significantly more difficult to study than free monoids and free groups. Indeed, before the explicit constructions by Scheiblich \cite{Scheiblich1972, Scheiblich1973} and the solution of the word problem by Munn \cite{Munn1974}, essentially nothing was known about free inverse monoids except their existence, which in turn was first noted in 1961 by Wagner \cite{Wagner1961}. Indeed, even the fact that the free inverse monoid on countably many generators embeds in the free inverse monoid on two generators, an obvious fact for free groups resp. monoids, is non-trivial to demonstrate \cite{Reilly1972}. One striking and easily formulated difficulty comes from the fact that, quite unlike free groups, no free inverse monoid of rank $n \geq 1$ is finitely presentable as a monoid \cite{Schein1975} (this result has very recently been extended from \textit{finitely presentable} to the even weaker property $\FP_2$ by Gray \& Steinberg \cite{Gray2021b}). Many other properties of free groups and monoids fail to transfer to free inverse monoids. For example, the \textit{Diophantine problem}, i.e. the problem of solving equations, in free inverse monoids is undecidable in general \cite{Rozenblat1985}, while the corresponding problems in free monoids and groups were solved by Makanin \cite{Makanin1977, Makanin1982}. See also \cite{Rozenblat1979, Calbrix1997}. 

Thus it should not, perhaps, come as a surprise to learn that the word problem for inverse monoids is, in general, rather difficult to solve. After Munn's solution to the problem for free inverse monoids, the first systematic study of the word problem for finitely presented inverse monoids was initiated by Stephen in his Ph.D. thesis \cite{Stephen1987}, via an algorithmic procedure analogous to Dehn's construction of the \textit{Gruppenbild}. This method, known as \textit{Stephen's procedure}, is now a cornerstone of the area \cite{Margolis1987, Stephen1990, Margolis1990, Stephen1991, Stephen1992, Stephen1993, Stephen1998}. For example, it was proved that for \textit{any} words $w_i \in (A \cup A^{-1})^\ast$ for $1 \leq i \leq n$, the monoid $\pres{Inv}{A}{w_i^2 = w_i \: (1 \leq i \leq n)}$ has decidable word problem \cite{Stephen1992}. Furthermore, if $e,f \in (A \cup A^{-1})^\ast$ are words equal to $1$ in the free group, then the word problem for $\pres{Inv}{A}{e=f}$ is decidable \cite{Margolis1993, Birget1994}, see also \cite{Silva1992, Margolis2005, Hermiller2010}. A natural question, seeking to extend the results by Magnus and Adian, is: do all special one-relation inverse monoids $\pres{Inv}{A}{w=1}$ have a decidable word problem? This question was long known to be difficult; Ivanov, Margolis \& Meakin \cite{Ivanov2001} proved that decidability of this problem, even only when $w$ is a reduced word, would imply decidability of the word problem for all one-relation monoids.\footnote{Their proof of this fact contains a (small) gap, see Nyberg-Brodda \cite[\S6.2]{NybergBrodda2021b} for a fix.}

In a recent twist, however, the answer turned out to be negative: Gray \cite{Gray2020} proved the existence of a special one-relation inverse monoid with undecidable word problem. One of the difficulties in attempting to mimic Adian's theory of special one-relation monoids in the inverse case comes from the present inability to understand the r\^ole of the group of units. Gray \& Ru\v{s}kuc \cite{Gray2021} have recently demonstrated much pathological behaviour for the group of units of special one-relation inverse monoids. Indeed, there is currently no known algorithm for computing the group of units, given a presentation $\pres{Inv}{A}{w=1}$. There are many open problems and mysteries in this area; this article represents an attempt to understand it better.

An overview of the article is as follows. First, in \S\ref{Sec:_Pieces_and_Benois_intro} we give some more detailed background on the group of units and invertible pieces of a special one-relation inverse monoid, and describe the \textit{Benois algorithm} introduced by Gray \& Ru\v{s}kuc, and its associated \textit{Benois conjectures}. In \S\ref{Sec:_OR_groups_classes}, we define four classes of one-relator groups, defined via special inverse monoid presentations, and study their properties. In \S\ref{Sec:_Ohareseque_family}, we prove the existence of an infinite family of special one-relation inverse monoids exhibiting the same pathological behaviour as the \textit{O'Hare monoid}. We also prove that the monoid $\pres{Inv}{a,b}{a^2b^2a^2bab=1}$ is, surprisingly, a group. Finally, in \S\ref{Sec:_Counterexample_Bob_Nik} we present a counterexample to a conjecture by Gray \& Ru\v{s}kuc, showing that the Benois algorithm does not always correctly compute the invertible pieces of a special one-relation inverse monoid. 

\clearpage

\section{Terminology and definitions}\label{Sec:_Pieces_and_Benois_intro}

\noindent The reader is assumed to be familiar with the theory of inverse monoids; for this, we refer the reader to Petrich \cite{Petrich1984}. In particular, an \textit{inverse} monoid $M$ is one satisfying the following law: for every $x \in M$, there exists a unique $y \in M$ such that $xyx = x$ and $yxy = y$. This unique ``pseudo-inverse'' $y$ of $x$ is denoted $x^{-1}$. The construction of the free inverse monoid on a set and the associated \textit{Wagner congruence} can be found in \cite[Chapter~VIII]{Petrich1984}. We will also assume the reader is familiar with the theory of presentations, e.g. from \cite{Adian1966, Magnus1966, Stephen1987}. We will denote monoid, group, resp. inverse monoid presentations by $\pres{Mon}{A}{R}, \pres{Gp}{A}{R}$ resp. $\pres{Inv}{A}{R}$. We refer the reader to Gray \& Ru\v{s}kuc \cite{Gray2021} for an in-depth background to the material. 

We will fix some notation. For a (finite) alphabet $A$, we will denote by $A^\ast$ the free monoid over $A$. This consists of all words over $A$ together with the operation of word concatenation, and its identity -- the empty word -- is denoted $1$. Equality of words in $A^\ast$, i.e. graphical equality, is denoted $\equiv$. We will associate to $A$ an alphabet of formal inverse symbols $A^{-1}= \{ a^{-1} \colon a \in A\}$, with $A \cap A^{-1} = \varnothing$. We let $\overline{A} = A \cup A^{-1}$. We say that a word $w \in \overline{A}^\ast$ is (freely) \textit{reduced} if it contains no subword of the form $aa^{-1}$ or $a^{-1}a$ for $a \in A$. We say that $w$ is \textit{cyclically} reduced if it is reduced and not of the form $aw'a^{-1}$ or $a^{-1}w'a$ for some $a \in A$ and $w' \in \overline{A}^\ast$. The free group on $A$ will be defined as the group with underlying set all reduced words in $\overline{A}$, and multiplication given by concatenation, then reducing the resulting word. These inverse symbols will also be used for the free inverse monoid. The following lemma will be used consistently (and implicitly) throughout this article.

\begin{lemma}[{E.g. \cite[Lemma~2.2(iii)]{Gray2021}}]
Let $w \in \overline{A}^\ast$, and let $I$ be an inverse monoid generated by $A$. Let $w'$ be the free reduction of $w$. If $w$ is right invertible in $I$, then $w =_I w'$. 
\end{lemma}

Here, a word is said to be (left/right) \textit{invertible} in a monoid if it represents a (left/right) invertible element of the monoid, see \S\ref{Subsec:Special_monoids}.

For two words $u, v \in \overline{A}^\ast$, we will let $[u, v] = uvu^{-1}v^{-1}$ denote their commutator. For a set $X \subseteq \overline{A}^\ast$ we let $\langle A \rangle$ denote the submonoid of the free group on $A$ generated by $X$. We let $X^\ast$ denote the submonoid of the free monoid $\overline{A}^\ast$ generated by $X$. Thus $\langle X \rangle$ is, in general, distinct from $X^\ast$ -- the former, for example, contains only reduced words. On the other hand, if $X \subseteq A^\ast$, then $\langle X \rangle = X^\ast$. For a word $w \in \overline{A}^\ast$, we let $\Pre(w)$ denote the set of \textit{prefixes} of $w$, i.e. if $w \equiv a_1 a_2 \cdots a_n$, where $a_i \in \overline{A}$, then 
\[
\Pre(w) = \Pre(a_1 a_2 \cdots a_n) = \{ a_1 a_2 \cdots a_j \mid 0 \leq j \leq n \},
\]
where $j=0$ is taken to mean that the word $a_1 a_2 \cdots a_j$ is empty. Thus, we have in particular $\varepsilon, w \in \Pre(w)$. A prefix of $w$ is \textit{proper} if it is not $w$ itself. Thus we will speak of e.g. the set of proper non-empty prefixes of a word. We define \textit{suffix} analogously. A word is \textit{self-overlap free} if none of its proper non-empty prefixes is also a suffix.

\subsection{Special monoids}\label{Subsec:Special_monoids}

We recall some of the terminology of special presentations. See \cite{Zhang1992b, Gray2021, NybergBrodda2020b} for further details. Let 
\begin{equation}\label{Eq:Special_presentation}
M = \pres{}{A}{w_1 = 1, w_2 = 1, \dots, w_k=1},
\end{equation}
where the presentation in (\ref{Eq:Special_presentation}) is either a monoid presentation or an inverse monoid presentation. Then $M$ is called \textit{special}. Thus we will speak of special inverse monoids and special monoids (and thereby implicitly assume that a presentation of the form given in (\ref{Eq:Special_presentation}) is simultaneously provided). The theory of special monoids is well-developed, originating in work by Adian \cite{Adian1960} and Makanin \cite{Makanin1966, Makanin1966b, NybergBrodda2021c}, and developed further by Zhang \cite{Zhang1991, Zhang1992a, Zhang1992b, Zhang1992c} and Nyberg-Brodda \cite{NybergBrodda2020b}, see also \cite[Chapters~3 \& 5]{MyThesis}. Of particular importance for special monoids is the notion of the \textit{minimal invertible pieces} of the presentation.  

We present this notion only for special inverse monoids, as it is well elaborated on for special (ordinary) monoids in the aforementioned references. Let $M$ be a special inverse monoid presentation with generators and relations as in (\ref{Eq:Special_presentation}). If $w \in \overline{A}^\ast$ represents a unit of $M$, i.e. an element $m \in M$ such that $mm^{-1} = m^{-1}m = 1$, then we say that $w$ is \textit{invertible}. We analogously define \textit{right} and \textit{left} invertible words. We say that an non-empty invertible word $u \in \overline{A}^\ast$ is \textit{minimal} if none of its non-empty proper prefixes is invertible. The set of minimal words forms a biprefix code as a subset of the free monoid $\overline{A}^\ast$. Every defining word $w_i$ for $1 \leq i \leq k$ is an invertible word, as $w_i =_M 1$. Thus it is easy to see that we can uniquely factorise every such word $w_i$ into minimal words. Note that this uses no fact about the structure of the monoid $M$, and the factorisation is not, in general, an effective one. Thus, for every $1 \leq i \leq k$, we uniquely factorise $w_i \equiv w_{i,1} w_{i,2} \cdots w_{i,\ell_i}$, where $w_{i,j}$ for $1 \leq j \leq \ell_i$ is a minimal word. The set of all minimal words arising in this way shall be denoted $\Lambda$, and called the set of \textit{presentation pieces} (or simply \textit{minimal invertible pieces}) of $M$. That is, 
\[
\Lambda = \bigcup_{i=1}^{k} \bigcup_{j=1}^{\ell_i} \{ w_{i,j} \} \subseteq \overline{A}^\ast.
\]
It can be shown that $\Lambda$ generates the subgroup consisting of all units of $M$ (see \cite[Proposition~4.2]{Ivanov2001}, for the stronger form here stated see \cite[Theorem~1.3]{Gray2021}). This subgroup is denoted $U(M)$, and is called the \textit{group of units} of $M$. For a special $k$-relation monoid 
\[
M=\pres{Mon}{A}{w_1=1, w_2 = 1, \dots, w_k=1}
\]
the group of units $U(M)$ is always a $k$-relator group, by Makanin \cite{Makanin1966}. In particular, as already proved by Adian \cite[Theorem~8]{Adian1966}, $U(M)$ is a one-relator group if $M$ is a special one-relation monoid. Although we do not directly study the group of units in this article, we highlight the fact that by contrast, in the inverse case $M = \pres{Inv}{A}{w=1}$, it need not be the case that $U(M)$ is a one-relator group; Gray \& Ru\v{s}kuc \cite[Theorem~7.1]{Gray2021} provide an example in which $U(M)$ is isomorphic to the free product of two copies of the fundamental group of a surface of genus $2$, which is not a one-relator group by \cite[Proposition~5.13]{Lyndon1977}. This demonstrates the contrast between special inverse and special ``ordinary'' monoids.

In general, computing the pieces of the presentation (\ref{Eq:Special_presentation}) is an undecidable problem; indeed, were it decidable, then one could decide whether or not $M$ is a group, for $M$ is clearly a group if and only if for every $a \in A$, we have $a\in \Lambda$ or $a^{-1} \in \Lambda$; but it is, in general, undecidable whether $M$ is a group. The corresponding statement, invertible pieces defined \textit{mutatis mutandis}, is true for special (ordinary) monoids. For special \textit{one-relation} monoids $\pres{Mon}{A}{w=1}$, on the other hand, there is such an algorithm; this is given by Adian's overlap algorithm, see \cite{Adian1966, Lallement1974}. We do not provide a full description of this (simple) algorithm here, but point out that if the defining word $w$ has no overlaps with itself (i.e. if no non-empty proper prefix of $w$ is also a suffix of $w$), then $\Lambda = \{ w \}$ (or $\Lambda = \varnothing$), i.e. there are no non-trivial pieces. For example, the group of units of $\pres{Mon}{a,b}{ababb = 1}$ is trivial, as $ababb$ has no overlaps with itself. 

A natural question is whether or not an analogue of Adian's algorithm can be applied to special one-relation \textit{inverse} monoids. This was shown not to be the case by Margolis, Meakin \& Stephen \cite{Margolis1987}, who considered the \textit{O'Hare monoid}
\[
\mathcal{O}  =\pres{Inv}{a,b,c,d}{abcdacdadabbcdacd=1} = \pres{Inv}{a,b,c,d}{r = 1}.
\]
Note that the defining word $r$ is (1) positive; and (2) self-overlap free. Were Adian's overlap algorithm applied to this word, it would hence conclude that the decomposition into minimal invertible factors of $r$ is trivial. However, using Stephen's procedure (see \cite{Stephen1987}), one can in fact show that the factorisation of $r$ into minimal invertible pieces is as
\begin{equation}\label{Eq:OHare_factors}
r \equiv (abcd)(acd)(ad)(abbcd)(acd),
\end{equation}
i.e. $\Lambda = \{ ad, acd, abcd, abbcd \}$, whereas the Adian algorithm would incorrectly yield $\Lambda = \{ abcdacdadabbcdacd \}$. We remark that the word problem for $\mathcal{O}$ was solved by Dolinka \& Gray \cite[Proposition~5.4]{Dolinka2021}. 

In view of the above example, if $w$ is a word which is (1) positive; and (2) self-overlap free; and (3) the factorisation into minimal invertible factors of $w$ in $\pres{Inv}{A}{w=1}$ is non-trivial, then we say that $\pres{Inv}{A}{w=1}$ is \textit{O'Haresque}. In \S\ref{Sec:_Ohareseque_family} we shall provide an infinite family of considerably simpler examples of O'Haresque monoids. In fact, we shall provide an infinite family of special one-relation inverse monoids $M = \pres{Inv}{A}{w=1}$ such that (1) and (2) hold, and furthermore such that (3') $M$ is a group. This is the first known example of a monoid with these three properties.

\subsection{The Benois algorithm}

In view of the O'Hare monoid, there was, for some time, no natural candidate for an algorithm to compute the minimal invertible pieces (or indeed the group of units) of a special inverse one-relation monoid. To amend this, Gray \& Ru\v{s}kuc \cite{Gray2021} recently introduced the \textit{Benois algorithm}. We now give a full description of this rather short algorithm here, as we shall frequently use it in this article. 

Let $M = \pres{Inv}{A}{w_i = 1 \: (i \in I)}$. Let 
\[
X = \bigcup_{i \in I} \left( \Pre(w_i) \cup \Pre(w_i^{-1})\right)  \subseteq \overline{A}^\ast.
\]
We call $X$ the set of \textit{Benois generators} associated to $M$, and we call the submonoid $\langle X \rangle$ of the free group on $\overline{A}^\ast$ generated by $X$ the \textit{Benois submonoid} associated to $M$. Every element of $\langle X \rangle$ is right invertible. By Benois' theorem \cite{Benois1969}, we can algorithmically test, for every prefix $p$ of some defining word $w_i$, whether $p^{-1}$ represents an element of $\langle X \rangle$. If it does, then $p^{-1}$ is right invertible; so $p$ is left invertible, and thereby also invertible. Thus, this gives a method for factorising the defining words into invertible pieces (although these need not be minimal). Explicitly, for every $i \in I$, we decompose 
\[
w_i \equiv w_{i,1} w_{i,2} \cdots w_{i,k_i}
\]
in such a way that for every proper prefix $p$ of $w_i$ we have 
\[
p^{-1}\in \langle X \rangle \iff p \equiv w_{i,1} w_{i,2} \cdots w_{i,j} \textnormal{ for some $1 \leq j \leq k_i$.}
\]

As mentioned earlier, the submonoid membership problem can be decided in any free group, so the Benois algorithm is an algorithm; indeed, Benois \cite{Benois1969} proved that the \textit{rational subset membership problem}, which generalises the submonoid membership problem, is decidable in any free group.\footnote{Ivanov, Margolis \& Meakin \cite[Lemma~5.4]{Ivanov2001} defined a submonoid analogue of Nielsen reductions for subgroups of free groups to give a direct proof of the decidability of the submonoid membership problem in free groups. However, their idea does not work as stated. In their terminology, the set $ X= \{ ab, b^{-1}, a\}$ is not N-reduced, as condition N2 is not satisfied with the triple $(V_1, V_2, V_3) = (ab, b^{-1}, a)$. If the proof were correct, then either (T1) or (T2) should thus be applicable to $X$; but neither operation is applicable. No proof of correctness is given (except by analogy with the subgroup case), and thus the issue seems at present unresolved. Nevertheless, Benois' result ensures that \cite[Lemma~5.4]{Ivanov2001}, if not its proof, remains true as stated. } The name of the Benois algorithm stems from the decidability of this latter problem. Gray \& Ru\v{s}kuc proved (see \cite[Theorem~4.5]{Gray2021}) that the Benois algorithm detects any factorisation that Adian's overlap algorithm does, and that furthermore it detects the factorisation (\ref{Eq:OHare_factors}) in the O'Hare monoid, thus strictly outperforming Adian's algorithm. 

Of course, given the undecidability of the problem, the Benois algorithm cannot possibly always correctly compute the minimal invertible pieces of a special inverse monoid. The following example shows that this failure can occur even in simple cases; in fact, the example (with $\operatorname{Mon}$ substituted for $\operatorname{Inv}$) was already given by Gray \& Ru\v{s}kuc \cite[Example~4.2]{Gray2021} for showing that the Adian algorithm does not always correctly compute the invertible pieces of a special monoid. 

\begin{example}
Consider the special three-relation inverse monoid 
\[
M = \pres{Inv}{a,b,c,d}{ab=1,cabd=1,cdd=1}.
\]
Then the Benois submonoid of the free group on $\{ a, b, c, d\}$ associated to $M$ is 
\begin{align*}
\langle X \rangle = \langle a, ab, c, ca, cab, &cabd, cd, cdd, b^{-1}, b^{-1}a^{-1}, d^{-1},  \\ & d^{-1}b^{-1}, d^{-1}b^{-1}a^{-1}, d^{-1}b^{-1}a^{-1}c^{-1}, d^{-1}d^{-1}c^{-1} \rangle. 
\end{align*}
Simplifying, we find 
\begin{align*}
\langle X \rangle = \langle a, ab, ca, cd, cab, cdd, cabd, b^{-1}a^{-1}, d^{-1}d^{-1}c^{-1}, d^{-1}b^{-1}a^{-1}c^{-1}\rangle.
\end{align*}
It is now a straightforward (but rather tedious!) task to verify that the inverse of no proper non-empty prefix of any defining relation are in this submonoid, e.g. by constructing an automaton as outlined in \cite[Remark~4.3]{Gray2021}. We leave this as an exercise for the interested reader. In particular, we find that the factorisation into invertible pieces discovered by the Benois algorithm is trivial. On the other hand, it is clear from $ab=1$ that $1 = cabd = cd$. Thus $d = cdd = 1$, so $d$ is invertible in $M$. Thus $c = cab = cabd = 1$, so $c$ is invertible in $M$. Thus $M \cong \pres{Inv}{a,b}{ab=1}$, the bicyclic monoid (Gray \& Ru\v{s}kuc, in their ordinary monoid example, incorrectly claim that all letters are invertible). 
\end{example}

Irrespective of the above example, and with the case of more than one relation put to the side, Gray \& Ru\v{s}kuc conjectured that the Benois algorithm always correctly computes the invertible pieces of a special \textit{one-relation} inverse monoid. In \S\ref{Sec:_Counterexample_Bob_Nik}, we present a counterexample to this conjecture.  

In light of this counterexample, we will introduce some more precise terminology. We let the \textit{positive, cyclically reduced}, resp. the \textit{reduced Benois conjecture} be the conjectures that the Benois algorithm correctly computes the factorisation into minimal invertible pieces of $w$ in $\pres{Inv}{A}{w=1}$ whenever $w$ is a \textit{positive, cyclically reduced}, resp. \textit{reduced} word. All three conjectures remain open, even accounting for the results in this article; the defining relation word $w$ of the counterexample given in \S\ref{Sec:_Counterexample_Bob_Nik} is neither positive, cyclically reduced, nor reduced. 

\clearpage
\section{Classes of one-relator groups}\label{Sec:_OR_groups_classes}

\noindent In this section, we will investigate four families of finitely generated one-relator groups, where the classes are defined via their properties as inverse monoids. Specifically, we will consider the classes $\pos, \cred, \red$ and $\any$, defined as in the following table. \\

\begin{center}
\begin{tabular}{| c | m{8cm} |}
\hline \rule{0pt}{1.5\normalbaselineskip}
 $\pos$ & One-relator groups $G$ s.t. $G \cong \pres{Inv}{A}{w=1}$ for some \textit{positive} word $w \in A^\ast.$\\ 
 \hline\rule{0pt}{1.5\normalbaselineskip}
 $\cred$ & One-relator groups $G$ s.t. $G \cong \pres{Inv}{A}{w=1}$ for some \textit{cyclically reduced} word $w \in \overline{A}^\ast$. \\  
\hline\rule{0pt}{1.5\normalbaselineskip} 
 $\red$ & One-relator groups $G$ s.t. $G \cong \pres{Inv}{A}{w=1}$ for some \textit{reduced} word $w \in \overline{A}^\ast$. 
 \\
 \hline\rule{0pt}{1.5\normalbaselineskip} 
 $\any$ & One-relator groups $G$ s.t. $G \cong \pres{Inv}{A}{w=1}$ for \textit{some} word $w \in \overline{A}^\ast$. 
 \\ \hline
\end{tabular}
\end{center}
\bigskip 

As every positive word is cyclically reduced, we have 
\[
\pos \subseteq \cred \subseteq \red \subseteq \any.
\]
The purpose of this section is to clarify and sharpen this sequence of inclusions. We first prove that $\any$ is the class of all one-relator groups. Next, we demonstrate that
\[
\pos \subset \cred \subset \any.
\]
Furthermore, if one additionally assumes the reduced Benois conjecture, then we demonstrate that $\red \subset \any$. 

Many questions remain. For example, is $\red=\cred$? The monoids $\pres{Inv}{A}{w=1}$ when $w \in \overline{A}^\ast$ is a cyclically reduced word are all $E$-unitary, but this need not be the case when $w$ is only reduced \cite{Ivanov2001}. Thus, there are special one-relation inverse monoids definable by a single reduced word which cannot be defined by a single cyclically reduced word. Nevertheless, it may be the case that $\red = \cred$, and that these differences vanish in the strong setting of groups (of course, all groups are $E$-unitary).

We shall begin by describing the relatively easy classes $\any$ and $\pos$.

\subsection{The class $\any$}

Of the four classes described, the easiest to describe is, in a sense, the class $\any$, by the following straightforward proposition.

\begin{proposition}\label{Prop:_OR_is_any}
Any one-relator group admits a special one-relation inverse monoid presentation. That is, the family $\any$ is precisely the class of all one-relator groups.
\end{proposition}
\begin{proof}
Let $G = \pres{Gp}{A}{w=1}$, where $w \in \overline{A}^\ast$ is any word. Enumerate the generating set $A$ as $\{ a_1, a_2, \dots, a_k \}$, and let 
\[
w' \equiv \bigg(\prod_{i=1}^k a_ia_i^{-1}a_i^{-1}a_i\bigg) w \bigg(\prod_{i=1}^k a_ia_i^{-1}a_i^{-1}a_i\bigg).
\]
Now $w' = w$ in the free group on $A$, as $a_i a_i^{-1}a_i a_i^{-1} = 1$ for $1 \leq i \leq k$. Let $I = \pres{Inv}{A}{w'=1}$. We claim that $I \cong G$. To show this, we will show that $a_j$ is invertible in $I$ for every $1 \leq j \leq k$. From this, it will follow that $I$ is a group; thus $I = \pres{Gp}{A}{w'=1}$, but from $w'=w$ in the free group on $A$, this latter group is just $\pres{Gp}{A}{w=1}$.

Note that for every $1 \leq j \leq k$, we have that
\[
\bigg( \prod_{i=1}^{j-1} a_i a_i^{-1}a_i^{-1} a_i \bigg) a_j 
\]
is a prefix of $w'$, and hence is right invertible in $I$; thus it is equal in $I$ to its free reduction in $I$. But this free reduction is simply $a_j$, and so $a_j$ is right invertible in $I$ for every $1 \leq j \leq k$. Analogously, for every $1 \leq j \leq k$, we have that
\[
\bigg( \bigg(\prod_{i=1}^{j-1} a_ia_i^{-1}a_i^{-1}a_i\bigg) a_j \bigg)^{-1} \equiv a_j^{-1} \bigg(\prod_{i=1}^{j-1} a_i^{-1}a_ia_ia_i^{-1}\bigg) 
\]
is an inverse of a suffix of $w'$, and hence is right invertible in $I$; thus it is equal in $I$ to its free reduction. But its free reduction is simply $a_j^{-1}$. We conclude that $a_j^{-1}$ is right invertible in $I$, and hence that $a_j$ is left invertible, for every $1 \leq j \leq k$. Thus $a_j$ is invertible for every $1 \leq j \leq k$, and the result follows. 
\end{proof}

Thus, the word $r \equiv \prod_{i=1}^k a_ia_i^{-1}a_i^{-1}a_i$ is a ``group-making'' word, in the sense that $\pres{Inv}{A}{rwr=1}$ is a group for any word $w \in \overline{A}^\ast$. We shall revisit, in a more refined fashion, the idea of ``group-making'' words in \S\ref{Sec:_Ohareseque_family}.

\subsection{The class $\pos$}

Perrin \& Schupp \cite{Perrin1984} proved that a one-relator group $\pres{Gp}{A}{w=1}$ admits a presentation $\pres{Mon}{A}{w'=1}$ if and only if the group is a positive one-relator group. We first note that their proof of this fact can be repeated line by line in the inverse monoid case, and utilising the fact that if $\pres{Mon}{A}{w=1}$ is a group, then so too is $\pres{Inv}{A}{w=1}$. That is, we have:

\begin{proposition}\label{Prop:_pos_is_pos}
A one-relator group admits a positive special one-relation inverse monoid presentation if and only if it is a positive one-relator group. 
\end{proposition}
\begin{proof}[Sketch]
The ``only if'' direction is obvious. For the other direction, consider $G = \pres{Gp}{A}{r=1}$ a positive one-relator group, with $r$ a positive word, and $A = \{ a, a_1, \dots, a_n \}$. By an elementary Nielsen transformation, we may assume that $r$ begins and ends with $a$. By utilising the free group automorphism defined by 
\[
\varphi(a) = a_1 a_2 \cdots a_n a a_n \cdots a_2 a_1 
\]
and $\varphi(a_i) = a_i$ for all $1 \leq i \leq n$, we find that $G$ is isomorphic to 
\begin{equation}\label{Eq:Pos_group_aut}
\pres{Gp}{a, a_1, a_2, \dots, a_n}{a_1 a_2 \cdots a_n w a_n \cdots a_2 a_1 = 1},
\end{equation} 
where $w$ is some positive word. It is clear that
\[
\pres{Inv}{a, a_1, a_2, \dots, a_n}{a_1 a_2 \cdots a_n w a_n \cdots a_2 a_1 = 1}
\] 
presents a group; and thus it presents the same group as \eqref{Eq:Pos_group_aut}, namely $G$. 
\end{proof}

Thus understanding the class $\pos$ is entirely reduced to a group-theoretic problem (free from any inverse monoids). As mentioned in the introduction, not every one-relator group is positive; for example, the group $\mathbb{Z}^2$ is not positive. Hence $\pos$ is strictly smaller than the class of all one-relator groups.

\subsection{The class $\red$} 

By contrast with the class $\any$, the class $\red$ is somewhat more difficult to approach directly. However, the reduced Benois conjecture would yield quite some insight into the class; in particular, we have:

\begin{proposition}\label{Prop:_red_is_different_from_any}
The reduced Benois conjecture implies that $\red \subset \any$.
\end{proposition}

To prove this, we give a quick, slightly technical, lemma.

\begin{lemma}\label{Lem:_prefixes_of_w=1_means_all}
Let $w \in \overline{A}^\ast$ and let $r \equiv w [a,b]w^{-1}$ be such that $r$ is reduced. Then any non-empty product $p \equiv p_1 p_2 \cdots p_n$ with $p_i \in \Pre(r) \cup \Pre(r^{-1})$ is such that $p=1$ if and only if $w \equiv \varepsilon$ and $p_i = [a,b]^{\pm 1}$ for all $1 \leq i \leq n$. 
\end{lemma}
\begin{proof}
The proof of the ``if'' direction is immediate. For the forward implication, we prove this by induction on $|w|$ and $n$. First, if $n=1$, then as $r$ is reduced, each of its (non-empty) prefixes does not equal $1$, so $p_1 \neq 1$. For the other base case, if $|w|=0$, then $r \equiv [a,b]$. Let $X =  \Pre(r) \cup \Pre(r^{-1})$. Then
\begin{align*}
\langle X\rangle &= \langle \{ a, ab, aba^{-1}, [a,b], b, ba, bab^{-1}, [b,a] \} \rangle\\
&= \langle a, b, [a,b], [b,a] \rangle = \langle a, b, [a,b], [a,b]^{-1} \rangle.
\end{align*}
By passing to the abelianisation, any element of $X^\ast$ equalling $1$ in $\langle X \rangle$ is now graphically a product of elements $[a, b]^{\pm 1}$, as required. 

Suppose $n>1$ and $|w|>0$, and let $p_1 p_2 \cdots p_n = 1$. As $w \not\equiv \varepsilon$, all $p_i$ now begin with the same letter, namely the first letter of $w$, as $r^{-1} \equiv w [b,a] w^{-1}$. Call this letter $x \in \overline{A}$. In particular $|x^{-1}w|<|w|$. First, assume every $p_i$ $(1 \leq i \leq n)$ ends with $x^{-1}$. Then if $s \equiv (x^{-1}w)[a,b](x^{-1}w)^{-1}$, we have that $x^{-1}p_ix \in \Pre(s) \cup \Pre(s^{-1})$. Furthermore, as $|w|>0$, $s$ is reduced. Since now $|s|<|r|$, and 
\begin{equation*}
p_1 p_2 \cdots p_n = 1 \quad \implies \quad (x^{-1}p_1x)(x^{-1}p_2x) \cdots (x^{-1}p_nx) = 1
\end{equation*}
we find a contradiction by the inductive hypothesis (on $|w|$).

Thus, suppose some $p_i$ ends with $y \in \overline{A}$ with $y \neq x^{-1}$. As $p_1 p_2 \cdots p_n = 1$, we may cyclically permute the $p_i$, if necessary, and assume without loss of generality that $p_n$ ends with $y$. When freely cancelling $p_1 p_2 \cdots p_n$, the last letter $y$ will thus not cancel with the first letter of $p_1$. In particular, $p_1 p_2 \cdots p_n$ has some suffix equal to $1$, and we have $p_1 p_2 \cdots p_i' = 1$ for some $1 \leq i \leq n$, with $p_i \equiv p_i' p_i''$ for $p_i', p_i'' \in \overline{A}^\ast$. As $r$ is reduced, so too is $p_n$, so $i<n$. But now $p_i' \in \Pre(r) \cup \Pre(r^{-1})$, so by the inductive hypothesis (on $n$) we have a contradiction. 
\end{proof}
\begin{remark}
The same method of proof as in the proof of Lemma~\ref{Lem:_prefixes_of_w=1_means_all}, which is adapted from a method due to Eberhard, shows that no non-empty product of non-empty prefixes of a reduced word in a free group equals $1$. In particular, if $w \in \overline{A}^\ast$ then $w^{-1}$ is not a product of prefixes of $w$ (for, if $p_1 p_2 \cdots p_n = w^{-1}$ for prefixes $p_i$ of $w$, then $p_1 p_2 \cdots p_n w = 1$). This was conjectured to be the case by Ivanov, Margolis \& Meakin \cite[p. 105]{Ivanov2001}, and slightly sharpens \cite[Theorem~5.3(c)]{Ivanov2001}.
\end{remark}

We are now able to prove, assuming the reduced Benois conjecture, that $\red$ is smaller than $\any$.

\begin{proof}[Proof of Proposition~\ref{Prop:_red_is_different_from_any}]
Let $G = \pres{Gp}{a, b}{[a, b] = 1} \cong \mathbb{Z}^2$. By Proposition~\ref{Prop:_OR_is_any}, $G$ admits a one-relator special inverse monoid presentation, i.e. $G \in \any$. Indeed, by the method in the proof of that proposition, we have 
\[
G \cong \pres{Inv}{a,b}{aa^{-1}a^{-1}abb^{-1}b^{-1}b[a,b]aa^{-1}a^{-1}abb^{-1}b^{-1}b = 1}.
\]
On the other hand, assuming the Benois conjecture holds for $\red$, then we claim we have $G \not\in \red$, which would complete the proof. Suppose for contradiction that there exists some reduced word $r \in \overline{A}^\ast$ such that $I = \pres{Inv}{A}{r=1}$ is such that $G \cong I$. Then $G \cong \pres{Gp}{A}{r=1}$. It follows that $|A| = 2$ and, if we write $A = \{ a, b \}$, that $r$ is conjugate to $[a,b]$ in the free group on $A$ (see \cite[\S6, p. 159]{Magnus1930}). Thus we can write $r \equiv w[a, b]w^{-1}$ for some reduced $w \in \overline{A}^\ast$. 

Let $X$ be the Benois generators associated to $I$, and let $p \in \overline{A}^\ast$ be a prefix of $r$ such that $p^{-1} \in \langle X \rangle$. Write $p^{-1} = p_1 p_2 \cdots p_n$ for $p_i \in X$, $1 \leq i \leq n$. Then $p_1 p_2 \cdots p_n p = 1$. As $X = \Pre(r) \cup \Pre(r^{-1})$, we deduce by Lemma~\ref{Lem:_prefixes_of_w=1_means_all} that $w \equiv \varepsilon$ and $p \equiv [a,b]$. Hence, assuming the reduced Benois conjecture, the factorisation of $r$ into minimal invertible pieces is trivial; there are no proper non-empty invertible prefixes of $r \equiv [a,b]$. We conclude that $I$ cannot be a group; a contradiction.
\end{proof}

Describing precisely which one-relator groups appear in $\red$ seems like an interesting problem. In particular, we do not know if $\red = \cred$. 

\subsection{The class $\cred$}

The class $\cred$ seems hard to describe in generality. We begin with a rather straightforward proposition. 

\begin{proposition}\label{Prop:ZxZ_is_not_cyclically_reduced_OSIM}
The one-relator group $\mathbb{Z}^2 = \pres{Gp}{a,b}{[a,b] = 1}$ does not admit a cyclically reduced special one-relation inverse monoid presentation. 
\end{proposition}
\begin{proof}
Suppose $\mathbb{Z}^2 \cong \pres{Inv}{A}{w=1}$ where $w$ is cyclically reduced. Then we also have $\mathbb{Z}^2 \cong \pres{Gp}{A}{w=1}$. As in the proof of  Proposition~\ref{Prop:_red_is_different_from_any}, we have $A = \{ a, b\}$ and $w$ is conjugate to $[a, b]$. Without loss of generality, assume $w \equiv [a,b]$. But $\pres{Inv}{a,b}{[a,b]=1}$ is not a group, as the homomorphism induced by $a \mapsto x$ and $b \mapsto 1$ is surjective onto $\pres{Inv}{x}{xx^{-1}=1}$, the bicyclic monoid. This is a contradiction. 
\end{proof}
\begin{remark}
The statement of Proposition~\ref{Prop:ZxZ_is_not_cyclically_reduced_OSIM} holds true even if the one-relator group in question is generalised to
\begin{equation}\label{Eq:Orientable_surface_group}
\Gamma_g = \pres{Gp}{a_1, b_1, \dots, a_g, b_g}{[a_1, b_1][a_2, b_2] \cdots [a_g, b_g] = 1},
\end{equation}
i.e. the fundamental group of a compact orientable $2$-manifold of genus $g>0$. This follows from the fact that, up to conjugacy, the only one-relator presentation for $\Gamma_g$ is that given in (\ref{Eq:Orientable_surface_group}), see \cite[Theorem~N10, p. 176]{Magnus1966}. Note that, of course, $\mathbb{Z}^2 \cong \Gamma_1$. 
\end{remark}

Thus there is a one-relator group which is not in $\cred$; in particular $\cred \subset \any$ by Proposition~\ref{Prop:_OR_is_any}.

\begin{lemma}\label{Lem:abwaab_is_a_group}
Let $A = \{ a, b \}$, and let $w \in \overline{A}^\ast$ be any word. Then the special inverse monoid $\pres{Inv}{a,b}{abwaab = 1}$ is a group.
\end{lemma}
\begin{proof}
We use the Benois algorithm, and consider the submonoid of the free group on $\{ a, b\}$ generated by the prefixes of the defining word, and the inverses of the suffixes of the same. It suffices to show that $a^{-1}$ and $b$ are elements of this submonoid, for then $a^{-1}$ is right invertible and consequently $a$ is left invertible; being a prefix of the defining word, $a$ is also right invertible, so $a$ will be invertible. Symmetrically, $b$ will be invertible. But 
\begin{align*}
(ab)(aab)^{-1} = aa^{-1}a^{-1} &= a^{-1}, \\
(ab)(aab)^{-1}(ab) = a^{-1}ab &= b,
\end{align*}
and, as $ab$ is a prefix of the defining word, and $aab$ is a suffix, it follows that the monoid is a group.
\end{proof}

Note that the statement of Lemma~\ref{Lem:abwaab_is_a_group} is also true for ordinary monoid presentations $\pres{Mon}{a,b}{abwaab=1}$ with $w \in A^\ast$. Indeed, using Adian's overlap algorithm from \cite[Chapter~III]{Adian1966}, one finds that either $ab$ is a piece, or else the monoid is a group. But if $ab$ is a piece, then $wa$ is invertible, so $a$ is left invertible; thus $ab$ cannot be a minimal invertible piece, a contradiction. This is an alternative proof of Lemma~\ref{Lem:abwaab_is_a_group}.

\begin{proposition}\label{Prop:Exists_nonpositive_CR_group}
There exists a cyclically reduced special one-relation inverse monoid which is not a positive one-relator group. In other words, we have $\pos \subset \cred$.
\end{proposition}
\begin{proof}
Let $G = \pres{Gp}{a,b}{a^{-1} = [a, bab^{-1}]}$. Then $G$ is not a positive one-relator group; indeed, it is not residually solvable, as $a^{-1} \neq_G 1$ by the \textit{Freiheitssatz}, but $a^{-1}$ obviously lies in every term of the derived series of $G$. However, every positive one-relator group is residually solvable \cite{Baumslag1971}. Thus $G \not\in \pos$ by Proposition~\ref{Prop:_pos_is_pos}. On the other hand, we claim that 
%\[
%G \cong \pres{Inv}{a,b}{ababab^{-1}a^{-1}a^{-1}b^{-1}a^{-1}a^{-1}b^{-1}ababaab = 1} := \pres{Inv}{a,b}{w=1}. 
%\]
\[
G \cong \pres{Inv}{a,b}{(ab)^2ab^{-1}(a^{-2}b^{-1})^2(ab)^2a^2b = 1} := \pres{Inv}{a,b}{w=1}. 
\]

As the defining word, which we denote by $w$, is cyclically reduced, we would hence have $G \in \cred$, yielding the claim. 

First, $\pres{Inv}{a,b}{w=1}$ is a group by Lemma~\ref{Lem:abwaab_is_a_group}, and therefore 
\[
\pres{Inv}{a,b}{w=1} \cong \pres{Gp}{a,b}{w=1}.
\]
Let $F_2$ denote the free group with basis $a, b$, and let $\varphi \colon F_2 \to F_2$ be the endomorphism of $F_2$ defined by $a \mapsto abaab$ and $b \mapsto ab$. We claim that $\varphi$ is an automorphism of $F_2$. Indeed, it is surjective as 
\begin{align*}
a = \varphi(ba^{-1}bb), \quad b = \varphi(b^{-1}ab^{-1}),
\end{align*}
and as $F_2$ is Hopfian, it follows that $\varphi$ is an automorphism. Note that $\varphi(bab^{-1})=ababa$. Hence 
\begin{align*}
\varphi([a, bab^{-1}]a) &= (abaab)(ababa)(abaab)^{-1}(ababa)^{-1}(abaab) \\
&= \underbrace{ab}_{w'} (aabab) \underbrace{abab^{-1}a^{-1}a^{-1}b^{-1}a^{-1}a^{-1}b^{-1}ab}_{w''} \\
&\equiv w'(aabab)w'',
\end{align*}
Importantly, note that $w \equiv abw''w'aab$. Thus $w$ is a cyclic conjugate of $w'aababw''$, so $\pres{Gp}{a,b}{w=1} = \pres{Gp}{a,b}{w'aababw''=1}$. Finally, as $\varphi$ is an automorphism, we hence have
\begin{align*}
G = \pres{Gp}{a,b}{[a, bab^{-1}]a = 1} &\cong \pres{Gp}{a,b}{\varphi([a, bab^{-1}]a) = 1} \\
&= \pres{Gp}{a,b}{w'(aabab)w'' = 1} \\
&= \pres{Gp}{a,b}{abw''w'aab = 1} \\
&\cong \pres{Inv}{a,b}{abw''w'aab =1} = \pres{Inv}{a,b}{w=1},
\end{align*}
where the penultimate isomorphism is, as mentioned, by Lemma~\ref{Lem:abwaab_is_a_group}.
\end{proof}
\begin{remark}
The group $G$ in the proof of Proposition~\ref{Prop:Exists_nonpositive_CR_group} is known as the \textit{Baumslag-Gersten group}. This group was introduced in a one-page paper by G. Baumslag \cite{Baumslag1969}, who proved that $G$ is not residually finite (indeed, all its finite quotients are cyclic). It has later been studied extensively \cite{Gersten1992, Myasnikov2011}.
\end{remark}

The following question is hence natural: 

\begin{question}
Is there a group-theoretic characterisation of the class $\cred$ of one-relator groups?
\end{question}

At present, this seems (to the author) like a rather difficult question. For example, the Dehn function of the Baumslag-Gersten group (which is in $\cred$) grows faster than any iterated tower of exponentials; the author is not aware of any such group in $\pos$. There are also natural connections, which we do not expand on, between this question and a refuted conjecture of Magnus' (see \cite[p. 401]{Magnus1966}, disproved by Zieschang \cite{Zieschang1970} and McCool \& Pietrowski \cite{McCool1971}).
\clearpage
\section{An infinite O'Haresque family}\label{Sec:_Ohareseque_family}

\noindent As mentioned in \S\ref{Sec:_Pieces_and_Benois_intro}, there is pathological behaviour regarding the group of units of the O'Hare monoid $\mathcal{O}$. Namely, $\mathcal{O}$ is such that its defining relation $r$ is a positive word which is self-overlap free, and yet $r$ factors non-trivially into minimal invertible factors as \eqref{Eq:OHare_factors}. Note that $\mathcal{O}$ is not a group. In this section, we give an even sharper example. Indeed, we also give the first example of a special one-relation inverse monoid which is (1) defined by a single, self overlap-free word, and (2) a group. Namely, let:
\[
\mathcal{I}_0 = \pres{Inv}{a,b}{a^2b^2a^2bab=1}.
\]
We shall prove the surprising fact that $\mathcal{I}_0$ is a group (Theorem~\ref{Theo:aabbaabab_is_group}). This is a significantly more straightforward (than $\mathcal{O}$) example of an O'Haresque monoid.

\begin{theorem}\label{Theo:aabbaabab_is_group}
The special inverse one-relation monoid $\mathcal{I}_0$ is a group.
\end{theorem}
\begin{proof}
Let $r \equiv a^2b^2a^2bab$. It suffices to show that $a$ is invertible, as from this, since $r$ is a positive word, it follows that any factorisation of $r$ into minimal invertible pieces must be such that every invertible piece containing $b$ also begins and ends with $b$; thus every invertible piece is either $a$ or $b$. We would thus conclude that $\mathcal{I}_0$ is a group. To show that $a$ is invertible, it suffices to show that it is left invertible; thus, by the Benois algorithm, it suffices to show $a^{-1} \in \langle X \rangle$, where $X = \Pre(r) \cup \Pre(r^{-1})$ is the set of Benois generators. But 
\begin{align}\label{Eq:aabbaabab_prod}
(b^{-1} a^{-1} b^{-1}a^{-2})\cdot (a^2b^2) \cdot (b^{-1}a^{-1}) \cdot (a^2b^2) \cdot (b^{-1}a^{-1}) = a^{-1}, 
\end{align}
as is readily verified, and every factor in the left-hand side of \eqref{Eq:aabbaabab_prod} is in $X$. Thus $a^{-1}$ is left invertible, and $\mathcal{I}_0$ is a group.
\end{proof}

In fact, it is not hard to show that $\mathcal{I}_0 \cong \pres{Gp}{a,b}{a^2 = b^3}$, the trefoil knot group. In the remainder of this section, we will prove that there exists an infinite family of $2$-generated O'Haresque monoids. Furthermore, we will show that, unlike the O'Hare monoid, every monoid in this family is a group. We begin with a simple lemma. 

\begin{lemma}\label{Lem:infinitelymanyolp}
The word $a^2b^2wa^3bab$ is self-overlap free for infinitely many choices of word $w \in \{ a, b\}^\ast$. 
\end{lemma}
\begin{proof}
Indeed, choosing $w \equiv a^i$ or $w \equiv b^i$ for $i \geq 0$ will always result in $a^2b^2wa^3bab$ being self-overlap free. 
\end{proof}

\begin{proposition}\label{Prop:aabbwaaabab_is_a_group}
Let $w \in \{ a, b \}^\ast$ be any positive word. Then the monoid $I_w = \pres{Inv}{a,b}{a^2b^2wa^3bab = 1}$ is a (one-relator) group.
\end{proposition}
\begin{proof}
We use the Benois algorithm. Let $X$ be the Benois generators of the presentation. Then we certainly have 
\begin{equation}\label{Eq:a2b2etc}
\Pre(a^2b^2) \cup \Pre(b^{-1}a^{-1}b^{-1}a^{-3}) \subseteq X.
\end{equation}
Let $Y$ be the set on the left-hand side in \eqref{Eq:a2b2etc}. Then $\langle Y \rangle \leq \langle X \rangle$. As in the proof of Theorem~\ref{Theo:aabbaabab_is_group}, it suffices to show $a^{-1} \in \langle Y \rangle$. But this is immediate, as 
\begin{align*}
(a^2b^2) \cdot (b^{-1}a^{-1}) \cdot (a^2b) \cdot (b^{-1}a^{-1}b^{-1}a^{-3}) = a^{-1}
\end{align*}
and the left-hand side is clearly a product of elements in $Y$.
\end{proof}

We conclude by Lemma~\ref{Lem:infinitelymanyolp} that we have proved the following.

\begin{theorem}
There exists an infinite family of O'Haresque special one-relation inverse monoids.
\end{theorem}

The above theorem (and the original O'Hare monoid) demonstrates that, unlike in the case of ordinary special monoids, the properties of having a self-overlap free defining word and being a group are not strongly linked for special one-relation inverse monoids -- or, if they are, then they are linked in some fundamentally different way. We end with an open question, which has a similar flavour to the questions asked in \S\ref{Sec:_OR_groups_classes} about $\pos, \cred, \red$ and $\any$. 

\begin{question}
Let $\ovl$ be the class of one-relator groups consisting of all $G$ such that $G$ is  isomorphic to $\pres{Inv}{A}{w=1}$ for some overlap-free word $w \in A^\ast$. Is there a group-theoretic characterisation of $\ovl$?
\end{question}

For example, one can easily check that $I_\varepsilon = \pres{Inv}{a,b}{a^2b^2a^3bab = 1}$ is free-by-cyclic. In particular it is residually finite. A natural question is whether there exists some $w \in A^\ast$ such that $I_w$ is not residually finite, or indeed whether every one-relator group in $\ovl$ is residually finite. We suspect, for the latter question, that this is not the case. 

Finally, note that the word $w$ in the statement of Proposition~\ref{Prop:aabbwaaabab_is_a_group} is arbitrary -- none of its properties, other than the fact that it is positive, are used in the proof. Thus the positive words $a^2b^2$ and $a^3bab$ can, in a sense, be seen as a pair of ``group-making words''; place them at the start resp. end of a given word, and the inverse monoid is a group. Of course, many other similar group-making words can be found, but ensuring that the resulting word is overlap-free (to ensure O'Haresquity) requires some care. We ask some elementary questions regarding group-making words.

\begin{question}\label{Quest:GMW}
Let $\mathfrak{G}(A)$ be the subset of $A^\ast \times A^\ast$ consisting of all pairs of words $(u, v)$ with the property that: for every word $w \in A^\ast$, the monoid $\pres{Inv}{A}{uwv=1}$ is a group. What are the language-theoretic properties of $\mathfrak{G}(A)$? Is it a rational subset of $A^\ast \times A^\ast$? If $u \in A^\ast$, does there always exist some $v \in A^\ast$ such that $(u, v) \in \mathfrak{G}(A)$?
\end{question}

Thus, for example, we have $(ab,ba), (a^2b^2, a^3bab) \in \mathfrak{G}(A)$. Note that when $|A|=2$, the answer to the last part of Question~\ref{Quest:GMW} is clearly ``yes'' for non-empty $u$.

\clearpage
\section{A counterexample to a conjecture by Gray \& Ru\v skuc}\label{Sec:_Counterexample_Bob_Nik}

\noindent Gray \& Ru\v skuc \cite{Gray2021} conjectured that the Benois algorithm always correctly computes the decomposition into minimal invertible pieces of the defining word of a special one-relation inverse monoid. We now refute, by way of counterexample, this conjecture.

\begin{theorem}\label{Thm:Counterexample_to_benois}
There exists a special one-relation inverse monoid 
\[
\pres{Inv}{A}{w=1}
\] such that the Benois algorithm does not correctly compute the factorisation of $w$ into minimal invertible pieces.
\end{theorem}
\begin{proof}
Let $A = \{ a, b\}$, and let $w \equiv abab^{-1}ba^{-1}b^{-1}$. Let $I = \pres{Inv}{A}{w=1}$. Then we claim that (1) the factorisation of $w$ into minimal invertible pieces in $I$ is as $(a)(bab^{-1})(ba^{-1}b^{-1})$; but that (2) the Benois algorithm factorises $w$ as $(a)(bab^{-1}ba^{-1}b^{-1})$. 

For (1), note that as the defining word $abab^{-1}ba^{-1}b^{-1}$ is right invertible, it is equal to its reduced form, which is $a$. Hence $a=1$ in $I$. It follows that also $bab^{-1}ba^{-1}b^{-1} = 1$ in $I$. Thus $b$ is right invertible in $I$, so $bb^{-1} = 1$. But then $bab^{-1} = 1$ in $I$, too. Hence $bab^{-1}$ (and similarly also $ba^{-1}b^{-1}$) is invertible. Any finer factorisation would imply that $b$ would be invertible, and hence $I$ would be the infinite cyclic group $\mathbb{Z}$. But this is impossible: the map $\varphi$ from $I$ to the bicyclic monoid $B = \pres{Inv}{x}{xx^{-1}=1}$ defined by $a \mapsto 1$ and $b \mapsto x$ extends to a surjective homomorphism, as $\varphi(abab^{-1}ba^{-1}b^{-1}) = xx^{-1}xx^{-1} = 1$ in $B$; but $B$ cannot be a quotient of $\mathbb{Z}$. One may also easily see that this is not the case using Stephen's procedure \cite{Stephen1987}. 

For (2), let $F_2$ be the free group on $\{ a, b \}$. We must solve the membership problem in the submonoid $\langle X \rangle$ of $F_2$ generated by the elements
\begin{align*}
X &= \{ a, ab, aba, abab^{-1}, abab^{-1}b, abab^{-1}ba^{-1}, abab^{-1}ba^{-1}b^{-1} \} \:\cup \\
&\:\cup \{ b, ba, bab^{-1}, bab^{-1}b, bab^{-1}ba^{-1}, bab^{-1}ba^{-1}b^{-1}, bab^{-1}ba^{-1}b^{-1}a^{-1}\}.
\end{align*}
Freely reducing all words, and simplifying the generating set, we have
\begin{align*}
\langle X \rangle = \langle  a, a^{-1}, b, bab^{-1} \rangle.
\end{align*}
Now every element of $\langle X \rangle$ is right invertible. As $a^{-1} \in \langle X \rangle$, the Benois algorithm tells us that $a^{-1}$ is right invertible; hence $a$ is left invertible and thereby also invertible, being a prefix of $w$. To check the other prefixes, one may employ a number of solutions to the membership problem for $\langle X \rangle$ in $F_2$. Indeed, the language $L$ of words in $\overline{A}^\ast$ representing an element of $\langle X \rangle$ is a regular language, as free reduction can be simulated by a monadic rewriting system (see \cite{Book1982}). A finite state automaton which accepts this language $L$ can be constructed using the method outlined in \cite[Remark~4.3]{Gray2021}, see also \cite[Theorem~4.1]{Kambites2007}. We find the automaton: 
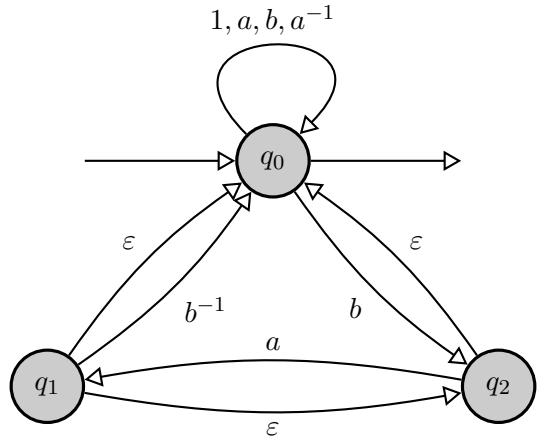
\begin{center}
\begin{tikzpicture}[initial text=,>=open triangle 60,thick,node distance=2cm,on grid,auto,accepting/.style=accepting by arrow,every state/.style={very thick,fill=black!20},every loop/.style={min distance=0mm,in=45,out=135,looseness=8}]

\node[state,initial,accepting,accepting distance=2cm,initial distance=2cm] (q_0) at (0,0) {$q_0$};
\node[state] (q_1) at (-3,-3) {$q_1$};
\node[state] (q_2) at (3,-3) {$q_2$};

\path[-open triangle 60, thick] 
(q_1) edge[bend right=10] node[swap] {$b^{-1}$} (q_0)
(q_1) edge[bend left=10] node [] {$\varepsilon$} (q_0)
(q_1) edge[bend right = 10] node [swap] {$\varepsilon$} (q_2)
(q_2) edge[bend right = 10] node [swap] {$a$} (q_1)
(q_0) edge[bend right = 10] node [swap] {$b$} (q_2)
(q_2) edge[bend right = 10] node [swap] {$\varepsilon$} (q_0)
(q_0) edge [loop] node {$1, a, b, a^{-1}$} ();
\end{tikzpicture}
\end{center}
A word represents an element of $\langle X \rangle$ if and only if it is accepted by this automaton; for example, $ba^nb^{-1}$ represents an element of $\langle X \rangle$ for every $n \geq 0$. We see that in the set $\{ a^{-1}, (ab)^{-1}, (aba)^{-1}, \dots \}$ of inverses of proper non-empty prefixes of $w$, only $a^{-1}$ is accepted by this automaton. In particular, $(abab^{-1})^{-1}\equiv ba^{-1}b^{-1}a^{-1}$ is not accepted, even though $abab^{-1}$ is invertible. Hence the factorisation into minimal invertible factors of $w$ obtained by the Benois algorithm is $(a)(bab^{-1}ba^{-1}b^{-1})$. This is, as we saw in (1), not the correct factorisation.
\end{proof}

An improved algorithm, which deals with the above counterexample, will appear in future work by the author. 

\section*{Acknowledgements}

The author thanks the Dame Kathleen Ollerenshaw Trust for funding his current research at the University of Manchester. Some of the research appearing in this article was carried out while the author was a Ph.D. student at the University of East Anglia. The author thanks Robert D. Gray, John Meakin, and Stuart Margolis for helpful discussions.

\bibliographystyle{plain}

{
\bibliography{onerelatorinverseunits.bib}

\begin{thebibliography}{10}

\bibitem{Adian1955}
S.~I. Adian.
\newblock Algorithmic unsolvability of problems of recognition of certain
  properties of groups.
\newblock {\em Dokl. Akad. Nauk SSSR (N.S.)}, 103:533--535, 1955.

\bibitem{Adian1960}
S.~I. Adian.
\newblock The problem of identity in associative systems of a special form.
\newblock {\em Soviet Math. Dokl.}, 1:1360--1363, 1960.

\bibitem{Adian1966}
S.~I. Adian.
\newblock Defining relations and algorithmic problems for groups and
  semigroups.
\newblock {\em Trudy Mat. Inst. Steklov.}, 85:123, 1966.

\bibitem{Baumslag1969}
Gilbert Baumslag.
\newblock A non-cyclic one-relator group all of whose finite quotients are
  cyclic.
\newblock {\em J. Austral. Math. Soc.}, 10:497--498, 1969.

\bibitem{Baumslag1971}
Gilbert Baumslag.
\newblock Positive one-relator groups.
\newblock {\em Trans. Amer. Math. Soc.}, 156:165--183, 1971.

\bibitem{Benois1969}
Mich\`ele Benois.
\newblock Parties rationnelles du groupe libre.
\newblock {\em C. R. Acad. Sci. Paris S\'{e}r. A-B}, 269:A1188--A1190, 1969.

\bibitem{Birget1994}
Jean-Camille Birget, Stuart~W. Margolis, and John~C. Meakin.
\newblock The word problem for inverse monoids presented by one idempotent
  relator.
\newblock {\em Theoret. Comput. Sci.}, 123(2):273--289, 1994.

\bibitem{Book1982}
Ronald~V. Book, Matthias Jantzen, and Celia Wrathall.
\newblock Monadic {T}hue systems.
\newblock {\em Theoret. Comput. Sci.}, 19(3):231--251, 1982.

\bibitem{Boone1958b}
William~W. Boone.
\newblock An analysis of {T}uring's ``{T}he word problem in semi-groups with
  cancellation''.
\newblock {\em Ann. of Math. (2)}, 67:195--202, 1958.

\bibitem{Boone1958}
William~W. Boone.
\newblock The word problem.
\newblock {\em Proc. Nat. Acad. Sci. U.S.A.}, 44:1061--1065, 1958.

\bibitem{Calbrix1997}
Hugues Calbrix.
\newblock La th\'{e}orie monadique du second ordre du mono\"{\i}de inversif
  libre est ind\'{e}cidable.
\newblock volume~4, pages 53--65. 1997.
\newblock Journ\'{e}es Montoises (Mons, 1994).

\bibitem{Church1937}
Alonzo Church.
\newblock Combinatory logic as a semi-group.
\newblock {\em Bulletin of the American Mathematical Society}, 43:333, 1937.

\bibitem{Dehn1911}
M.~Dehn.
\newblock \"{U}ber unendliche diskontinuierliche {G}ruppen.
\newblock {\em Math. Ann.}, 71(1):116--144, 1911.

\bibitem{Dehn1914}
M.~Dehn.
\newblock Die beiden {K}leeblattschlingen.
\newblock {\em Math. Ann.}, 75(3):402--413, 1914.

\bibitem{Dolinka2021}
Igor Dolinka and Robert~D. Gray.
\newblock New results on the prefix membership problem for one-relator groups.
\newblock {\em Trans. Amer. Math. Soc.}, 374(6):4309--4358, 2021.

\bibitem{Dyck1882}
Walther Dyck.
\newblock Gruppentheoretische {S}tudien.
\newblock {\em Math. Ann.}, 20(1):1--44, 1882.

\bibitem{Gersten1992}
S.~M. Gersten.
\newblock Dehn functions and {$l_1$}-norms of finite presentations.
\newblock In {\em Algorithms and classification in combinatorial group theory
  ({B}erkeley, {CA}, 1989)}, volume~23 of {\em Math. Sci. Res. Inst. Publ.},
  pages 195--224. Springer, New York, 1992.

\bibitem{Gray2020}
Robert~D. Gray.
\newblock Undecidability of the word problem for one-relator inverse monoids
  via right-angled {A}rtin subgroups of one-relator groups.
\newblock {\em Invent. Math.}, 219(3):987--1008, 2020.

\bibitem{Gray2021}
Robert~D. Gray and Nik Ru\v{s}kuc.
\newblock On groups of units of special and one-relator inverse monoids.
\newblock {\em Pre-print}, 2021.
\newblock Available at arXiv:2103.02995.

\bibitem{Gray2021b}
Robert~D. Gray and Benjamin Steinberg.
\newblock Free inverse monoids are not $\operatorname{FP}_2$.
\newblock {\em Comptes Rendus de l`Academie des Sciences - Series I:
  Mathematics}, 359:1047--1057, 2021.

\bibitem{Hermiller2010}
Susan Hermiller, Steven Lindblad, and John Meakin.
\newblock Decision problems for inverse monoids presented by a single sparse
  relator.
\newblock {\em Semigroup Forum}, 81(1):128--144, 2010.

\bibitem{Ivanov2001}
S.~V. Ivanov, S.~W. Margolis, and J.~C. Meakin.
\newblock On one-relator inverse monoids and one-relator groups.
\newblock {\em J. Pure Appl. Algebra}, 159(1):83--111, 2001.

\bibitem{Kambites2007}
Mark Kambites, Pedro~V. Silva, and Benjamin Steinberg.
\newblock On the rational subset problem for groups.
\newblock {\em J. Algebra}, 309(2):622--639, 2007.

\bibitem{Lallement1974}
G\'{e}rard Lallement.
\newblock On monoids presented by a single relation.
\newblock {\em J. Algebra}, 32:370--388, 1974.

\bibitem{Lyndon1977}
Roger~C. Lyndon and Paul~E. Schupp.
\newblock {\em Combinatorial group theory}.
\newblock Ergebnisse der Mathematik und ihrer Grenzgebiete. Springer-Verlag,
  Berlin-New York, 1977.

\bibitem{Magnus1932}
W.~Magnus.
\newblock Das {I}dentit\"{a}tsproblem f\"{u}r {G}ruppen mit einer definierenden
  {R}elation.
\newblock {\em Math. Ann.}, 106(1):295--307, 1932.

\bibitem{Magnus1930}
Wilhelm Magnus.
\newblock \"{U}ber diskontinuierliche {G}ruppen mit einer definierenden
  {R}elation. ({D}er {F}reiheitssatz).
\newblock {\em J. Reine Angew. Math.}, 163:141--165, 1930.

\bibitem{Magnus1966}
Wilhelm Magnus, Abraham Karrass, and Donald Solitar.
\newblock {\em Combinatorial group theory: {P}resentations of groups in terms
  of generators and relations}.
\newblock Interscience Publishers [John Wiley \& Sons, Inc.], New
  York-London-Sydney, 1966.

\bibitem{Makanin1966b}
G.~S. Makanin.
\newblock {\em On the {I}dentity {P}roblem for {F}initely {P}resented Groups
  and {S}emigroups}.
\newblock PhD thesis, Steklov Mathematical Institute, Moscow, 1966.

\bibitem{Makanin1966}
G.~S. Makanin.
\newblock On the identity problem in finitely defined semigroups.
\newblock {\em Dokl. Akad. Nauk SSSR}, 171:285--287, 1966.

\bibitem{Makanin1977}
G.~S. Makanin.
\newblock The problem of the solvability of equations in a free semigroup.
\newblock {\em Mat. Sb. (N.S.)}, 103(145)(2):147--236, 319, 1977.

\bibitem{Makanin1982}
G.~S. Makanin.
\newblock Equations in a free group.
\newblock {\em Izv. Akad. Nauk SSSR Ser. Mat.}, 46(6):1199--1273, 1344, 1982.

\bibitem{Margolis1990}
S.~W. Margolis, J.~C. Meakin, and J.~B. Stephen.
\newblock Free objects in certain varieties of inverse semigroups.
\newblock {\em Canad. J. Math.}, 42(6):1084--1097, 1990.

\bibitem{Margolis2005}
Stuart~W. Margolis, John Meakin, and Zoran {\v{S}}uni\'{k}.
\newblock Distortion functions and the membership problem for submonoids of
  groups and monoids.
\newblock In {\em Geometric methods in group theory}, volume 372 of {\em
  Contemp. Math.}, pages 109--129. Amer. Math. Soc., Providence, RI, 2005.

\bibitem{Margolis1993}
Stuart~W. Margolis and John~C. Meakin.
\newblock Inverse monoids, trees and context-free languages.
\newblock {\em Trans. Amer. Math. Soc.}, 335(1):259--276, 1993.

\bibitem{Margolis1987}
Stuart~W. Margolis, John~C. Meakin, and Joseph~B. Stephen.
\newblock Some decision problems for inverse monoid presentations.
\newblock In {\em Semigroups and their applications ({C}hico, {C}alif., 1986)},
  pages 99--110. Reidel, Dordrecht, 1987.

\bibitem{Markov1947b}
A.~Markov.
\newblock The impossibility of certain algorithms in the theory of associative
  systems. {II}.
\newblock {\em Doklady Akad. Nauk SSSR (N.S.)}, 58:353--356, 1947.

\bibitem{Markov1947a}
A.~Markov.
\newblock On the impossibility of certain algorithms in the theory of
  associative systems.
\newblock {\em C. R. (Doklady) Acad. Sci. URSS (N.S.)}, 55:583--586, 1947.

\bibitem{Markov1951}
A.~A. Markov.
\newblock The impossibility of algorithms for the recognition of certain
  properties of associative systems.
\newblock {\em Doklady Akad. Nauk SSSR (N.S.)}, 77:953--956, 1951.

\bibitem{McCool1971}
James McCool and Alfred Pietrowski.
\newblock On free products with amalgamation of two infinite cyclic groups.
\newblock {\em J. Algebra}, 18:377--383, 1971.

\bibitem{Munn1974}
W.~D. Munn.
\newblock Free inverse semigroups.
\newblock {\em Proc. Lond. Math. Soc. (3)}, 29:385--404, 1974.

\bibitem{Myasnikov2011}
Alexei Myasnikov, Alexander Ushakov, and Dong~Wook Won.
\newblock The word problem in the {B}aumslag group with a non-elementary {D}ehn
  function is polynomial time decidable.
\newblock {\em J. Algebra}, 345:324--342, 2011.

\bibitem{Novikov1952}
P.~S. Novikov.
\newblock On algorithmic unsolvability of the problem of identity.
\newblock {\em Doklady Akad. Nauk SSSR (N.S.)}, 85:709--712, 1952.

\bibitem{Novikov1958b}
P.~S. Novikov and S.~I. Adian.
\newblock Das {W}ortproblem f\"{u}r {H}albgruppen mit einseitiger
  {K}\"{u}rzungsregel.
\newblock {\em Z. Math. Logik Grundlagen Math.}, 4:66--88, 1958.

\bibitem{MyThesis}
C.-F. Nyberg-Brodda.
\newblock {\em {The Word Problem and Combinatorial Methods for Groups and
  Semigroups}}.
\newblock PhD thesis, University of East Anglia, UK, 2021.

\bibitem{NybergBrodda2021b}
Carl-Fredrik Nyberg-Brodda.
\newblock The word problem for one-relation monoids: a survey.
\newblock {\em Semigroup Forum}, 103(2):297--355, 2021.

\bibitem{NybergBrodda2021c}
Carl-Fredrik Nyberg-Brodda.
\newblock {A translation of G. S. Makanin's 1966 Ph.D. thesis "On the Identity
  Problem for Finitely Presented Groups and Semigroups"}.
\newblock February 2021.
\newblock Available online at arXiv:2102.00745.

\bibitem{NybergBrodda2020b}
Carl-Fredrik Nyberg-Brodda.
\newblock On the word problem for special monoids.
\newblock {\em Pre-print}, November 2020.
\newblock Available online at arXiv:2011.09466.

\bibitem{Otto1988}
Friedrich Otto.
\newblock An example of a one-relator group that is not a one-relation monoid.
\newblock {\em Discrete Math.}, 69(1):101--103, 1988.

\bibitem{Perrin1984}
Dominique Perrin and Paul Schupp.
\newblock Sur les mono\"{\i}des \`a un relateur qui sont des groupes.
\newblock {\em Theoret. Comput. Sci.}, 33(2-3):331--334, 1984.

\bibitem{Petrich1984}
Mario Petrich.
\newblock {\em Inverse semigroups}.
\newblock Pure and Applied Mathematics (New York). John Wiley \& Sons, Inc.,
  New York, 1984.
\newblock A Wiley-Interscience Publication.

\bibitem{Post1947}
Emil~L. Post.
\newblock Recursive unsolvability of a problem of {T}hue.
\newblock {\em J. Symbolic Logic}, 12:1--11, 1947.

\bibitem{Rabin1958}
Michael~O. Rabin.
\newblock Recursive unsolvability of group theoretic problems.
\newblock {\em Ann. of Math. (2)}, 67:172--194, 1958.

\bibitem{Reilly1972}
N.~R. Reilly.
\newblock Free generators in free inverse semigroups.
\newblock {\em Bull. Austral. Math. Soc.}, 7:407--424, 1972.

\bibitem{Rozenblat1979}
B.~V. Rozenblat.
\newblock Positive theories of free inverse semigroups.
\newblock {\em Sibirsk. Mat. Zh.}, 20(6):1282--1293, 1408, 1979.

\bibitem{Rozenblat1985}
B.~V. Rozenblat.
\newblock Diophantine theories of free inverse semigroups.
\newblock {\em Sibirsk. Mat. Zh.}, 26(6):101--107, 190, 1985.

\bibitem{Scheiblich1972}
H.~E. Scheiblich.
\newblock Free inverse semigroups.
\newblock {\em Semigroup Forum}, 4:351--359, 1972.

\bibitem{Scheiblich1973}
H.~E. Scheiblich.
\newblock Free inverse semigroups.
\newblock {\em Proc. Amer. Math. Soc.}, 38:1--7, 1973.

\bibitem{Schein1975}
B.~M. Schein.
\newblock Free inverse semigroups are not finitely presentable.
\newblock {\em Acta Math. Acad. Sci. Hungar.}, 26:41--52, 1975.

\bibitem{Silva1992}
Pedro~V. Silva.
\newblock Rational languages and inverse monoid presentations.
\newblock {\em Internat. J. Algebra Comput.}, 2(2):187--207, 1992.

\bibitem{Stephen1987}
J.~B. Stephen.
\newblock {\em Applications of automata theory to presentations of monoids and
  inverse monoids}.
\newblock 1987.
\newblock Thesis (Ph.D.)--The University of Nebraska - Lincoln.

\bibitem{Stephen1990}
J.~B. Stephen.
\newblock Presentations of inverse monoids.
\newblock {\em J. Pure Appl. Algebra}, 63(1):81--112, 1990.

\bibitem{Stephen1991}
J.~B. Stephen.
\newblock The word problem for inverse monoids and related questions.
\newblock In {\em Monoids and semigroups with applications ({B}erkeley, {CA},
  1989)}, pages 129--143. World Sci. Publ., River Edge, NJ, 1991.

\bibitem{Stephen1992}
J.~B. Stephen.
\newblock Contractive presentations: a family of inverse monoids and semigroups
  with finite {$\mathscr{R}$}-classes.
\newblock {\em Semigroup Forum}, 44(2):255--270, 1992.

\bibitem{Stephen1993}
J.~B. Stephen.
\newblock Inverse monoids and rational subsets of related groups.
\newblock {\em Semigroup Forum}, 46(1):98--108, 1993.

\bibitem{Stephen1998}
J.~B. Stephen.
\newblock Amalgamated free products of inverse semigroups.
\newblock {\em J. Algebra}, 208(2):399--424, 1998.

\bibitem{Thue1914}
Axel Thue.
\newblock Problem \"uber {V}er\"anderungen von {Z}eichenreihen nach gegebenen
  {R}egeln.
\newblock {\em Christiana {V}idenskaps-{S}elskabs {S}krifter, {I.} {M}ath.
  naturv. {K}lasse}, 10, 1914.

\bibitem{Turing1950}
A.~M. Turing.
\newblock The word problem in semi-groups with cancellation.
\newblock {\em Ann. of Math. (2)}, 52:491--505, 1950.

\bibitem{Wagner1961}
V.~V. Wagner.
\newblock Generalised heaps and generalised groups with transitive
  compatibility relation.
\newblock {\em Uchen. zap. Saratov. univ.}, 70:25--39, 1961.

\bibitem{Zhang1991}
Louxin Zhang.
\newblock Conjugacy in special monoids.
\newblock {\em J. Algebra}, 143(2):487--497, 1991.

\bibitem{Zhang1992b}
Louxin Zhang.
\newblock Applying rewriting methods to special monoids.
\newblock {\em Math. Proc. Cambridge Philos. Soc.}, 112(3):495--505, 1992.

\bibitem{Zhang1992c}
Louxin Zhang.
\newblock A short proof of a theorem of {A}djan.
\newblock {\em Proc. Amer. Math. Soc.}, 116(1):1--3, 1992.

\bibitem{Zhang1992a}
Louxin Zhang.
\newblock Some properties of finite special string-rewriting systems.
\newblock {\em J. Symbolic Comput.}, 14(4):359--369, 1992.

\bibitem{Zieschang1970}
Heiner Zieschang.
\newblock \"{U}ber die {N}ielsensche {K}\"{u}rzungsmethode in freien
  {P}rodukten mit {A}malgam.
\newblock {\em Invent. Math.}, 10:4--37, 1970.

\end{thebibliography}
}
\end{document}